\documentclass[11pt,reqno]{amsart}
\usepackage[T1]{fontenc}
\usepackage{lmodern}
\usepackage{amsmath,amssymb,amsthm,mathtools}
\usepackage[expansion=false]{microtype}
\usepackage{aliascnt}
\usepackage[hidelinks]{hyperref}
\usepackage[capitalize,nameinlink,noabbrev]{cleveref}
\numberwithin{equation}{section}
\newtheorem{theorem}{Theorem}[section]
\newaliascnt{proposition}{theorem}
\newtheorem{proposition}[proposition]{Proposition}
\aliascntresetthe{proposition}
\newaliascnt{lemma}{theorem}
\newtheorem{lemma}[lemma]{Lemma}
\aliascntresetthe{lemma}
\theoremstyle{remark}
\newaliascnt{remark}{theorem}
\newtheorem{remark}[remark]{Remark}
\aliascntresetthe{remark}
\newaliascnt{example}{theorem}
\newtheorem{example}[example]{Example}
\aliascntresetthe{example}
\crefname{theorem}{Theorem}{Theorems}
\crefname{proposition}{Proposition}{Propositions}
\crefname{lemma}{Lemma}{Lemmas}
\crefname{remark}{Remark}{Remarks}
\crefname{example}{Example}{Examples}
\Crefname{theorem}{Theorem}{Theorems}
\Crefname{proposition}{Proposition}{Propositions}
\Crefname{lemma}{Lemma}{Lemmas}
\Crefname{remark}{Remark}{Remarks}
\Crefname{example}{Example}{Examples}
\newcommand{\R}{\mathbb R}
\newcommand{\C}{\mathbb C}
\newcommand{\Sc}{\mathbb S}
\DeclareMathOperator{\sech}{sech}
\DeclareMathOperator{\diam}{diam}
\newcommand{\cc}{C_c^\infty}
\newcommand{\dd}{\,\mathrm d}

\usepackage{xcolor}

\makeatletter
\@namedef{subjclassname@2020}{2020 Mathematics Subject Classification}
\makeatother

\title{Concave entire solutions of the Liouville equation}
\author{Changfeng Gui}
\address{Department of Mathematics, University of Macau, Macau SAR, P. R. China}
\address{Zhuhai UM Science and Technology Research Institute, Hengqin, Guangdong 519031, P. R. China}
\email{changfeng.gui@um.edu.mo}
\author{Qinfeng Li}
\address{School of Mathematics, Hunan University, Changsha 410082, P. R. China}
\email{liqinfeng@hnu.edu.cn}
\author{Chunhua Wang}
\address{School of Mathematics and Statistics, and Key Lab NAA--MOE, Central China Normal University, Wuhan 430079, P. R. China}
\email{chunhuawang@ccnu.edu.cn}
\author{Lu Xu}
\address{School of Mathematics, Hunan University, Changsha 410082, P. R. China}
\email{xulu@hnu.edu.cn}
\subjclass[2020]{35B06, 35B35, 35J61}

\keywords{Liouville equation, concave entire solutions, one-dimensional symmetry, stability, dimension reduction}
\date{}

\begin{document}
	\begin{abstract}
			We prove that every concave entire solution of the Liouville equation $-\Delta u=e^{2u}$ in $\R^n$ is one-dimensional, and we extend the result to a class of general semilinear equations under suitable conditions on the nonlinearity. In dimension two, we also provide a PDE-based proof of the known classification of quasiconcave solutions to planar Liouville solutions, without relying on complex-analytic methods developed in the previous work by Bergweiler-Eremenko-Langley \cite{BEL}.
	\end{abstract}
	\maketitle
	
	\section{Introduction}
	The Liouville equation
	\begin{equation}\label{eq:Liouville}
		-\Delta u=e^{2u}\qquad\text{in }\R^n
	\end{equation}
	has a distinguished place among semilinear elliptic equations. In two
	dimensions it is the equation for a conformal metric of constant positive
	Gaussian curvature: if $g_u=e^{2u}|dz|^2$ on $\C$, then under this metric, the Gaussian curvature reads $K_{g_u}=-e^{-2u}\Delta u=1$. The same equation in higher dimensions is
	not the scalar-curvature equation for a conformal metric, but is a natural
	entire-space elliptic equation with an exponential source. A basic issue in
	both settings is to understand which global geometric restrictions on $u$
	force an otherwise flexible family of entire solutions to become rigid.
	Our concern is the particularly strong restriction of \emph{concavity},
	and, in the plane, the weaker condition of \emph{quasiconcavity}.
	
	\subsection{Classical classifications of planar Liouville solutions}
	The two-dimensional equation has a long history originating in Liouville's
	representation of constant-curvature metrics. In terms of a developing map,
	every smooth entire solution on $\C$ can be represented as
	\begin{equation}\label{eq:developing}
		e^{u(z)}=\frac{2|\Phi'(z)|}{1+|\Phi(z)|^2},
	\end{equation}
	where $\Phi:\C\to\widehat\C$ is a locally univalent meromorphic function;
	its poles, if any, are simple \cite{Liouville}. Thus the planar equation
	is linked to the spherical derivative of a meromorphic function and to
	the pullback of the round metric on $\Sc^2$. The representation is an
	extremely powerful starting point, but it does not by itself yield a
	classification: locally univalent meromorphic maps can have complicated
	behavior at infinity. One therefore seeks additional global conditions,
	such as finite conformal area, upper bounds, controlled growth, or
	geometric information about level sets and the induced metric.
	
	A classical landmark is the finite-area classification of Chen and Li
	\cite{CL}. If $u$ solves \eqref{eq:Liouville} in $\R^2$ and
	$\int_{\R^2}e^{2u}\dd x<\infty$, then, for some $\lambda>0$ and
	$x_0\in\R^2$,
	\begin{equation}\label{ql:eq:spherical}
		u(x)=\log\frac{2\lambda}{1+\lambda^2|x-x_0|^2}.
	\end{equation}
	These spherical solutions have total conformal area $4\pi$. The
	classification reflects the interaction between asymptotic information,
	the method of moving planes, and the underlying conformal geometry; see
	also Chou and Wan \cite{CW}. However, finite area is not a suitable
	hypothesis for every symmetry problem. For example,
	\begin{equation}\label{eq:standard-profile}
		u(s,t)=\log\sech s
	\end{equation}
	is an entire solution whose superlevel sets are parallel strips and
	whose conformal area in the plane is infinite.
	
	A substantial enlargement of the known families was obtained by
	Eremenko, Gui, Li and Xu \cite{EGLX}. They classified all planar
	solutions bounded from above and, more generally, investigated solutions
	with controlled logarithmic growth of the form
	$\limsup_{|z|\to\infty}u(z)/\log|z|<\infty$. Up to the natural changes
	of variables, the bounded-above solutions are either spherical or of
	the form
	\begin{equation}\label{eq:bounded-family}
		u_b(s,t)=\log\frac{2e^s}{1+b^2+2be^s\cos t+e^{2s}},
		\qquad b\ge0.
	\end{equation}
	For $b=0$ this is precisely the one-dimensional solution \eqref{eq:standard-profile}, while the
	solutions with $b>0$ depend genuinely on both variables. Cai, Lai and
	Zhang \cite{CLZ} subsequently obtained another proof of the bounded-above
	classification using a conformal-area growth estimate. A broader
	account of classifications under asymptotic and geometric hypotheses is
	provided by Cai and Lai \cite{CaiLaiSurvey}. Extensions of Liouville
	rigidity to surfaces under curvature or asymptotic assumptions appear in
	\cite{CaiLaiSurface,CFP,Ou,CFG}. A broader semilinear rigidity approach on manifolds, including
	an exponential equation in dimension two, is developed in
	\cite{CatinoMonticelli}.
	In particular, classification theorems based on logarithmic lower
	bounds at infinity concern a different asymptotic regime from
	\eqref{eq:standard-profile}, which decays linearly along its
	distinguished direction.
	
	Another source of rigidity is the intrinsic geometry of the metric
	$g_u$. Gui and Li \cite{GL} proved that
	$\diam_{g_u}(\R^2)\ge\pi$ and computed the diameter for
	\eqref{eq:bounded-family} as $\pi+2\arctan b$. They conjectured that the least diameter solutions must be either radial or one-dimensional. Bergweiler, Eremenko
	and Langley \cite{BEL} subsequently verified the least
	diameter rigidity conjecture and, more importantly, they also classified \emph{all} concave and quasiconcave planar solutions without
	assuming an upper bound or the attainment of a maximum. Here
	quasiconcavity means that every superlevel set $\{u\ge c\}$ is convex.
	Their classification shows that a quasiconcave planar solution is either
	spherical or one-dimensional; in particular, a concave solution must
	be one-dimensional.
	
	It is worth explaining the analytic content of \cite{BEL}, as it
	clarifies both the strength and the limitations of a different approach.
	After introducing the developing map $\Phi$ in
	\eqref{eq:developing}, their proof of the concavity theorem separates
	meromorphic maps according to their growth and employs the geometry of
	exceptional asymptotic values, decay in angular sectors, and estimates
	for positive harmonic functions. The stronger quasiconcavity theorem
	requires additional complex-analytic machinery. The Schwarzian
	derivative of $\Phi$ leads to an equation
	\begin{equation}\label{eq:schwarzian-ode}
		w''+A(z)w=0,
	\end{equation}
	with entire coefficient $A$. For polynomial $A$, the argument uses
	sectorial asymptotics and Stokes geometry; for transcendental $A$, it
	invokes Wiman--Valiron theory and asymptotic integration of the complex
	ODE to construct large violations of midpoint quasiconcavity. In fact,
	\cite{BEL} obtains a quantitative midpoint-defect conclusion in addition
	to the classification. These methods are finely adapted to the complex
	structure of the planar equation and yield information well beyond a
	bare symmetry statement. On the other hand, the dependence on a
	meromorphic developing map makes them difficult to transfer directly
	to higher dimensions or to a general nonlinearity $f(u)$.
	
	These results motivate a complementary PDE approach. We do not
	attempt to recover the full complex-analytic conclusions of
	\cite{BEL}, notably their quantitative midpoint-defect and metric
	diameter results. Our aim is instead to isolate mechanisms that are
	intrinsic to the differential equation: recession directions of convex
	sets, positive solutions of the linearized equation, spectral
	comparison, and, in two dimensions, a divergence-free tensor. This
	approach gives a short independent planar classification argument,
	but its main advantage is that the concavity method extends to every
	dimension and to nonlinearities without a developing-map
	representation.
	
	\subsection{The higher-dimensional concavity problem}
	For $n\ge3$, equation \eqref{eq:Liouville} no longer has the direct
	interpretation as a constant-Gaussian-curvature equation. It nevertheless
	has a background beyond planar conformal geometry. Exponential elliptic
	equations arise, for instance, in the classical theory of self-gravitating
	isothermal spheres, going back to Emden \cite{Emden}; see also \cite{Hunter}.
	They also occur naturally in the study of stability and Morse index for
	semilinear elliptic problems, where the analysis of entire solutions
	provides information relevant to corresponding problems on bounded
	domains \cite{Dancer2008,Farina,DF}.
	
	The motivation for the particular \emph{concavity} assumption is more
	specific, and comes from a striking connection with De Giorgi's
	conjecture for the Allen--Cahn equation. As observed in
	\cite{EGLX}, the standard one-dimensional Liouville profile
	\[
	h(t)=\log\sech t,\qquad -h''=e^{2h},
	\]
	has the remarkable property that, after a rescaling and a change of
	sign, its derivative is exactly the Allen--Cahn layer:
	\begin{equation}\label{eq:profile-connection}
		q(s)=-h'(s/\sqrt2)=\tanh(s/\sqrt2),
		\qquad q''=q^3-q.
	\end{equation}
	Moreover,
	\[
	q'(s)=-\frac1{\sqrt2}h''(s/\sqrt2)>0.
	\]
	Thus the strict monotonicity of the Allen--Cahn layer is precisely the
	strict concavity of the underlying Liouville profile, expressed at
	one higher order of differentiation.
	
	De Giorgi's conjecture \cite{DG} asks whether a bounded entire solution
	$v$ of $\Delta v=v^3-v$ that is monotone in one direction,
	$\partial_e v>0$, must be one-dimensional (at least in dimensions
	$n\le8$). The identity \eqref{eq:profile-connection} suggests a
	natural counterpart for the Liouville equation: if the sign condition
	on a \emph{first} directional derivative is replaced by a sign
	condition on the \emph{second} derivatives, does one again obtain
	one-dimensional symmetry? In several variables, the intrinsic
	formulation of this second-order sign condition is
	\[
	D^2u\le0,
	\]
	namely, concavity.  This leads to the question of whether \emph{every
		concave entire Liouville solution must depend on only one Euclidean
		variable}. The concavity hypothesis is thus suggested by the exact relation
	between the two one-dimensional profiles, rather than chosen ad hoc.
	This profile identity does not, however, transform general Liouville
	solutions into Allen--Cahn solutions.
	
	The dimension dependence makes this question especially interesting.
	The De Giorgi conjecture was proved in dimensions two and three by
	Ghoussoub--Gui \cite{GG} and Ambrosio--Cabr\'e \cite{AC}; Savin
	\cite{Savin} established it through dimension eight under prescribed
	limits at infinity, while del Pino--Kowalczyk--Wei \cite{DPKW}
	constructed counterexamples in dimensions at least nine. More
	recently, the works \cite{LLWWWW,CFFFS} have established the full
	monotone result in dimension four. By contrast, as we show here, the
	Liouville counterpart based on concavity holds in \emph{every}
	dimension. The analogy is thus the source of the question, not a
	prediction of the same dimension threshold.
	
	There is also a useful distinction from existing stability theory for
	exponential elliptic equations. Farina \cite{Farina} ruled out stable
	entire solutions of $-\Delta v=e^v$ when $2\le n\le9$, and
	Dancer--Farina \cite{DF} showed that there are no finite-Morse-index
	entire solutions when $3\le n\le9$; in dimension two, such solutions
	are necessarily radial. Stable radial entire solutions do occur for
	$n\ge10$ \cite{JL,Farina}. These facts neither imply nor contradict
	concavity rigidity. In particular, a nonconstant one-dimensional
	Liouville profile, when lifted to $\R^n$ for $n\ge2$, has infinite
	Morse index. Indeed, after finding a compactly supported negative
	test function as in \cref{lem:unstable}, one can place arbitrarily
	many disjoint translations of it in directions transverse to the
	profile. Their span is a negative subspace of arbitrarily large
	dimension. Consequently,
	finite-Morse-index classification does not settle the concavity
	question, and general stable-solution nonexistence is unavailable in
	high dimensions.
	
	The higher-dimensional problem was addressed by
	Eremenko--Gui--Li--Xu \cite[Theorem~1.10]{EGLX}, who proved that a
	concave entire Liouville solution is one-dimensional provided it has
	a local maximum. For a concave function such a point is automatically
	a global maximizer. Their proof uses the propagation of Hessian
	degeneracy through the constant-rank theory of Caffarelli--Friedman
	\cite{CF} and Korevaar--Lewis \cite{KL}; see also Bian--Guan
	\cite{BG} and the later refinements of Li--Xu \cite{LX}. The
	hypothesis that a maximum exists is substantial: a concave function
	on the whole space may have an unattained finite supremum or may be
	unbounded above. Constant-rank results alone do not rule out these
	possibilities. Controlling the supremum at infinity is therefore an obstruction not resolved by the earlier method.
	
	The following result removes that obstruction in every dimension.
	In particular, a finite attained maximum and the global
	one-dimensional structure are both consequences of concavity.
	\begin{theorem}\label{thm:Liouville}
		Let $n\ge1$. Every concave $C^2$ solution of \eqref{eq:Liouville}
		in $\R^n$ has the form
		\begin{equation}\label{eq:form}
			u(x)=\log\!\bigl(\lambda\sech(\lambda(x\cdot e-a))\bigr),
			\qquad e\in\Sc^{n-1},\quad\lambda>0,\quad a\in\R.
		\end{equation}
	\end{theorem}
	In particular, the existence of a finite attained maximum is a \emph{consequence} of the equation
	and concavity, not a hypothesis.
	
	\subsection{General nonlinearities and a planar energy identity}
	The mechanism behind the preceding theorem is not specific to the
	exponential. Consider the more general semilinear equation
	\begin{equation}\label{eq:semilinear}
		-\Delta u=f(u)\qquad\text{in }\R^n,
	\end{equation}
	where $f$ is positive. Two features of $f$ appear naturally in our
	argument. At large negative values, the source must decay fast enough
	to contradict linear recession of a concave function in more than one
	effective dimension. At large positive values, an eigenvalue estimate
	must prevent the values of monotone translates from escaping to
	$+\infty$. To formulate the latter condition, let
	\begin{equation}\label{eq:mu}
		\mu_c:=\inf_{t>0}\frac{f(c+t)}{t},\qquad c\in\R.
	\end{equation}
	We obtain the following all-dimensional extension.
	\begin{theorem}\label{thm:general}
		Suppose that $f\in C^1(\R)$ satisfies
		\begin{equation}\label{eq:assumptions}
			f>0,\qquad f'>0,\qquad
			\lim_{s\to-\infty}|s|f(s)=0,\qquad
			\mu_c>0\quad\text{for every }c\in\R.
		\end{equation}
		Then, for every $n\ge1$, each concave $C^2$ solution of
		\eqref{eq:semilinear} in $\R^n$ is one-dimensional.
	\end{theorem}
	The conditions at the two ends of the real line have distinct roles.
	The negative-tail hypothesis is sharp at the scale $|s|^{-1}$, as shown
	by the strictly concave non-one-dimensional solution in
	\cref{ex:critical-tail}. The positive-tail condition is equivalent, for
	positive continuous $f$, to
	\begin{equation}\label{eq:positive-tail}
		\liminf_{s\to+\infty}\frac{f(s)}{s}>0.
	\end{equation}
	Indeed, for fixed $c$, the quotient $f(s)/(s-c)$ tends to infinity as
	$s\downarrow c$ and has the same lower limit at infinity as $f(s)/s$.
	This assumption is automatic if $f$ is convex with $f'>0$, but no
	convexity of $f$ is used. Thus the result applies to a genuinely larger
	class of nonlinearities than those with convex source terms.
	
	The planar case admits a complementary argument without monotonicity
	of $f$ or any condition at positive infinity. The appropriate
	assumption is integrability of the negative tail, which is equivalent
	to the existence of a primitive bounded below. Owing to a special
	identity for the planar Hessian, this leads to a theorem under
	hypotheses not comparable with those above.
	\begin{theorem}\label{thm:planar}
		Let $f\in C^1(\R)$ satisfy $f>0$ and
		\begin{equation}\label{eq:planar-assumption}
			\int_{-\infty}^{0}f(s)\dd s<\infty.
		\end{equation}
		Then every concave $C^2$ solution of $-\Delta u=f(u)$ in $\R^2$
		is one-dimensional.
	\end{theorem}
	The assumptions of \cref{thm:general,thm:planar} are not nested. To
	indicate why the planar result is so short, let $F'=f$ and introduce
	the energy
	\[
	P=\frac12|\nabla u|^2+F(u).
	\]
	The Bochner identity gives
	\[
	\Delta P=|D^2u|^2-(\Delta u)^2
	=-2\det D^2u\le0
	\]
	under concavity in two dimensions. Since $F$ is bounded below, $P$ is
	a bounded-below superharmonic function on $\R^2$, hence constant.
	This forces the Hessian to have rank one at regular points and makes
	the regular level lines parallel. The argument is related to the long
	tradition of elliptic $P$-function methods, including those of
	Weinberger \cite{Weinberger}, Payne--Philippin \cite{PaynePhilippin}
	and Modica \cite{Modica}; the more recent rigidity results in
	\cite{CFP,Ou,CFG} also use auxiliary functions, though with different
	identities and asymptotic assumptions. The energy has an additional
	fluid-mechanical interpretation: $\nabla^\perp u$ is a steady
	incompressible Euler field and $-P$ is its pressure, up to a constant.
	In this language the conclusion means that the flow is a parallel
	shear; compare \cite{GXX,GLXX}.
	
	\subsection{A PDE route to the quasiconcave classification}
	The all-dimensional argument assumes concavity. In the plane, one can
	return to the weaker geometric condition of convex superlevel sets.
	As recalled above, Bergweiler--Eremenko--Langley \cite{BEL} already
	proved the complete quasiconcave classification. What is new here is
	a proof which starts from the PDE and never uses the developing map,
	the Schwarzian differential equation \eqref{eq:schwarzian-ode}, or the
	bounded-above classification in \cite{EGLX}. For completeness, we
	state the known classification in the notation of this paper.
	\begin{theorem}\label{thm:quasiconcave}
		Let $u\in C^2(\R^2)$ be a quasiconcave solution of
		$-\Delta u=e^{2u}$ in $\R^2$. Then $u$ is either a spherical solution
		\eqref{ql:eq:spherical} or a one-dimensional solution of the form
		\begin{equation}\label{ql:eq:oned}
			u(x)=\log\!\left(\lambda\sech\bigl(\lambda(x\cdot e-a)\bigr)\right),
			\qquad\lambda>0,\quad e\in\Sc^1,\quad a\in\R.
		\end{equation}
		Conversely, both families are quasiconcave.
	\end{theorem}
	Our first observation is that quasiconcavity itself, combined with the
	positive-potential linearized equation, rules out loss of the maximum
	at infinity. This also permits a short deduction using the previously
	known bounded-above classification. More importantly, a separate proof
	uses convexity of the level sets to derive
	\[
	|\nabla u|^2+e^{2u}\le e^{2\sup u}.
	\]
	Together with a trace-free and divergence-free tensor associated with
	the Liouville equation, this provides a fully PDE-based route to the
	explicit spherical and one-dimensional families. The tensor method is
	particularly elementary once the gradient estimate is known. Unlike
	\cite{BEL}, however, this proof does not recover the quantitative
	midpoint-defect assertion. Moreover, its final divergence-free identity
	is special to the exponential $e^{2u}$, so it should not be confused
	with the general semilinear theorem above.
	
	\subsection{Stability and the dimension-reduction argument}
	The link with De Giorgi's conjecture explains why concavity is a
	natural hypothesis. The proof, however, requires a different use of
	monotonicity and stability. For the general equation
	$-\Delta u=f(u)$, differentiation in a direction $e$ yields
	\[
	-\Delta u_e=f'(u)u_e.
	\]
	If $u_e>0$, this positive solution of the linearized equation is a
	Jacobi field and makes $u$ stable in the variational sense. This is
	the same basic monotonicity-to-stability mechanism that plays an
	important role in the Allen--Cahn theory \cite{GG,AC,AAC}.
	There is, however, a decisive spectral difference: whereas the
	one-dimensional Allen--Cahn layer is stable, every entire
	one-dimensional solution of $-\Delta u=f(u)$ with $f'>0$ is unstable.
	Thus a stable \emph{lower-dimensional limit}, if shown by induction
	to be one-dimensional, gives a contradiction. The role of stability
	here is to exclude  an unanticipated escape of concave solutions,	not to classify arbitrary stable solutions in high dimensions.
	
	Two different forms of escape to infinity must be distinguished.
	The first is geometric: a superlevel set of a concave function may
	be unbounded, even when the function attains its maximum, as happens
	for $\log\sech x_1$ in $\R^n$. Convexity supplies a recession
	direction, and concavity promotes it to a global nondecreasing
	direction. The second is analytic: when $u_e>0$, the translates
	$u(\,\cdot+te)$ increase, but their limit might a priori be
	$+\infty$, in which case no lower-dimensional elliptic equation
	could be extracted. Neither issue is resolved by constant-rank
	theory alone.
	
	The preliminary results of \cref{sec:preliminaries} address these
	obstructions separately. A convex-geometric lemma yields a direction
	with $u_e\ge0$ from an unbounded superlevel set. The linearized
	equation then gives the alternative $u_e\equiv0$ or $u_e>0$.
	In the latter case a ground-state identity supplies stability,
	whereas a cutoff argument shows that every one-dimensional
	profile is unstable. Finally, a Dirichlet eigenvalue comparison
	bounds the size of balls contained in a fixed superlevel set.
	This last estimate, and only this estimate, uses the positive-tail
	hypothesis on $f$.
	
	These tools lead to two complementary parts of the proof. In
	\cref{sec:maximum} we assume that the maximum is attained. Every
	nondecreasing direction must then be invariant, because $u_e$
	vanishes at a maximum point. After splitting off these invariant
	directions, the remaining function has a bounded superlevel set;
	concavity forces at least linear decay, and a spherical-mean
	argument excludes effective dimension greater than one. This is
	where the negative-tail assumption on $f$ is used. By contrast,
	\cref{sec:descent} deals with solutions whose supremum is not
	attained. A nontrivial monotonicity direction produces a stable
	family of translates. The eigenvalue bound rules out escape to
	$+\infty$, so elliptic compactness yields a finite concave limit in
	one fewer variable. Stability passes to that limit, and induction
	makes it one-dimensional, contradicting the instability established
	in \cref{sec:preliminaries}. This separation makes clear why the
	two assumptions at $-\infty$ and $+\infty$ enter the general
	semilinear theorem for different reasons.
	
	After completing this all-dimensional argument, we present a sharp
	negative-tail counterexample in \cref{sec:scope} and the independent planar $P$-function
	proof in \cref{sec:planar}. The final section,
	\cref{sec:quasiconcave}, returns to the quasiconcave Liouville
	classification and provides two proofs, including one independent of
	complex analysis and of earlier classification theorems.
	
	\section{Monotonicity directions and spectral preliminaries}\label{sec:preliminaries}
	Throughout the proofs, $B_R(x)$ denotes the open Euclidean ball of
	radius $R$ centered at $x$, and $B_R=B_R(0)$. A one-dimensional
	function on $\R^m$ has the form $u(x)=h(x\cdot e)$ for a fixed
	$e\in\Sc^{m-1}$. We collect three elementary ingredients whose roles
	will remain distinct: a recession-direction lemma for concave
	functions, a stability obstruction for one-dimensional solutions,
	and a spectral bound excluding large balls in a superlevel set.
	Together they allow us to handle both an attained maximum and the
	absence of a maximum with the same geometric starting point.
	
	A convex set that is both closed and unbounded contains a ray. For
	a superlevel set of a concave function, this observation has a
	stronger consequence: the direction of the ray makes the function
	nondecreasing at every point in space. This fact will be used
	whether or not the supremum of the solution is attained.
	
	\begin{lemma}\label{lem:recession}
		Let $u$ be a finite differentiable concave function on $\R^m$. If $\{u\ge c\}$ is nonempty and unbounded for some $c\in\R$, then there is $e\in\Sc^{m-1}$ such that
		\begin{equation}\label{eq:recession}
			u(x+te)\ge u(x)\quad\text{for all }x\in\R^m,\ t\ge0;
			\qquad u_e\ge0.
		\end{equation}
		If $u$ does not attain its supremum, every nonempty superlevel set is unbounded.
	\end{lemma}
	
	\begin{proof}
		Fix $x_0\in K:=\{u\ge c\}$ and choose $x_j\in K$ with $R_j=|x_j-x_0|\to\infty$. After taking a subsequence, $e_j=(x_j-x_0)/R_j$ converges to some $e\in\Sc^{m-1}$. For fixed $s\ge0$ and $R_j\ge s$, convexity of $K$ gives $x_0+se_j\in K$. Since $K$ is closed, $x_0+se\in K$ for every $s\ge0$.
		
		This ray gives a monotone direction for the whole function, including points outside $K$. Indeed, for fixed $x\in\R^m$, $t\ge0$ and $s>t$, concavity gives
		\[
		u\bigl((1-t/s)x+(t/s)(x_0+se)\bigr)
		\ge(1-t/s)u(x)+(t/s)c.
		\]
		Letting $s\to\infty$ proves \eqref{eq:recession}. If the supremum is not attained, each nonempty superlevel has level strictly below $\sup u$. Were one of them bounded, continuity and compactness would give a maximum on it that is also a global maximum, a contradiction.
	\end{proof}
	
	The conclusion has two different consequences. If a solution attains
	its maximum, a nondecreasing direction must be an invariant
	direction, because its derivative vanishes at the maximum and obeys
	the linearized equation. If the maximum is not attained, the same
	geometric lemma still supplies a monotone direction, which
	 may be
	strictly monotone and so calls for a stability argument. 
	
	Recall that $u$ is a stable solution to \eqref{eq:semilinear} if the following quadratic term is nonnegative for all test functions.
	\begin{equation}\label{eq:quadratic-form}
		Q_u[\eta]=\int_{\R^m}\bigl(|\nabla\eta|^2-f'(u)\eta^2\bigr)\dd x,
		\qquad\eta\in\cc(\R^n).
	\end{equation}
It is a standard fact that if $\psi>0$ solves $-\Delta\psi=f'(u)\psi$, then $u$ is a stable solution. Indeed, testing with $\eta^2/\psi$ gives the ground-state identity
	\begin{equation}\label{eq:groundstate}
		Q_u[\eta]=\int_{\R^m}\psi^2
		\left|\nabla\left(\frac{\eta}{\psi}\right)\right|^2\dd x\ge0.
	\end{equation}
	In particular, a positive directional derivative $u_e$ makes $u$ stable. If $u_e\ge0$ and $f'>0$, then
	\begin{equation}\label{eq:jacobi}
		-\Delta u_e=f'(u)u_e\ge0,
		\qquad u_e\equiv0\quad\text{or}\quad u_e>0
	\end{equation}
	by the strong minimum principle.
	
	The important point is that such a positive Jacobi field gives stability of the solution, and the stability property will be inherited by a
	translation limit in a lower-dimensional space. The next lemma shows
	why such a limit becomes impossible once a dimension-reduction
	argument has made it one-dimensional: the positivity of $f'$ alone
	forces a negative direction for the quadratic form. 
	
	\begin{lemma}\label{lem:unstable}
		Suppose $f\in C^1(\R)$ and $f'>0$. Every entire one-dimensional solution of \eqref{eq:semilinear} is unstable in any ambient dimension.
	\end{lemma}
	
	\begin{proof}
		First consider $-h''=f(h)$ on $\R$. Choose $0\le\chi\in\cc(\R)$ with $\chi=1$ on $[-1,1]$, and put $\zeta_R(s)=\chi(s/R)$. Then
		\[
		\int_{\R}|\zeta_R'|^2\dd s=R^{-1}\|\chi'\|_2^2\longrightarrow0,
		\qquad
		\int_{\R}f'(h)\zeta_R^2\dd s\ge\int_{-1}^{1}f'(h)\dd s>0
		\]
		for $R\ge1$. Hence $Q_h[\zeta_R]<0$ for all sufficiently large $R$. No integrability of $f'(h)$ on the whole line is needed.
		
		Fix $\zeta\in\cc(\R)$ with $Q_h[\zeta]<0$. For $m\ge2$, write $x=(s,y)\in\R\times\R^{m-1}$ and test the lifted solution with $\eta_L(s,y)=\zeta(s)\rho(y/L)$, where $0\ne\rho\in\cc(\R^{m-1})$. Direct integration gives
		\begin{equation}\label{eq:instability-product}
			Q_u[\eta_L]
			=L^{m-1}\|\rho\|_2^2Q_h[\zeta]
			+L^{m-3}\|\nabla\rho\|_2^2\|\zeta\|_2^2<0
		\end{equation}
		for large $L$.
	\end{proof}
	
	Instability of one-dimensional profiles will eventually close the
	induction, but it does not guarantee that monotone translates have a
	finite limit. In fact, if a solution grows without bound in a
	monotonicity direction, a fixed superlevel set could contain
	arbitrarily large balls. The following elementary eigenvalue
	comparison excludes precisely this possibility. It is the only
	preliminary result that requires the positive-tail condition
	$\mu_c>0$.
	
	\begin{lemma}\label{lem:ball}
		Let $f$ be continuous and positive, and suppose $\mu_c>0$ for some $c\in\R$, where $\mu_c$ is defined as in \eqref{eq:mu}. If a classical solution of $-\Delta u=f(u)$ satisfies $B_R(x_0)\subset\{u>c\}$, then
		\begin{equation}\label{eq:ball-bound}
			\mu_c\le\lambda_1^D(B_R)
			=R^{-2}\lambda_1^D(B_1),
		\end{equation}
		where $\lambda_1^D$ is the first eigenvalue of the Dirichlet Laplacian. In particular, a fixed superlevel set cannot contain balls of arbitrarily large radius.
	\end{lemma}
	
	\begin{proof}
		In the ball, $z=u-c>0$ satisfies $-\Delta z=f(c+z)\ge\mu_cz$. For $\eta\in\cc(B_R(x_0))$, integration by parts and completion of the square yield
		\[
		\mu_c\int_{B_R}\eta^2
		\le\int_{B_R}\nabla z\cdot\nabla\!\left(\frac{\eta^2}{z}\right)
		=\int_{B_R}|\nabla\eta|^2
		-\int_{B_R}\left|\nabla\eta-\frac{\eta}{z}\nabla z\right|^2
		\le\int_{B_R}|\nabla\eta|^2.
		\]
		Taking the Rayleigh infimum proves \eqref{eq:ball-bound}.
	\end{proof}

	The three ingredients are now in place. Geometric recession
	produces a monotonicity direction; the linearized equation converts
	strict monotonicity into stability, with one-dimensional instability
	as the eventual obstruction; and the spectral estimate prevents a
	monotone limit from diverging.
	
	\section{Rigidity when the supremum is attained}\label{sec:maximum}
	Assume first that a concave solution attains its supremum at some
	point $x_0$. The issue is not the existence of a preferred point,
	but the possibility that its high superlevel sets remain unbounded.
	Such behavior already occurs for the one-dimensional solution
	$\log\sech x_1$ lifted to $\R^m$, where the unbounded directions
	are exactly those along which the function is constant.
	
	The recession lemma together with the strong minimum principle
	identifies and removes these invariant directions, one at a time.
	What remains has a bounded superlevel set and hence decays at least
	linearly in every direction. If the remaining dimension were at
	least two, averaging over spheres would give a strictly decreasing,
	concave radial mean whose derivative nevertheless tends to zero.
	The negative-tail assumption rules out this incompatibility; no
	positive-tail growth condition is needed. We formulate the argument
	for general increasing nonlinearities in the following proposition.
	
	\begin{proposition}\label{prop:max}
		Let $f\in C^1(\R)$ satisfy
		\[
		f>0,\qquad f'>0,\qquad
		\lim_{s\to-\infty}|s|f(s)=0.
		\]
		If a concave entire solution of $-\Delta u=f(u)$ in $\R^m$ attains its supremum, then it is one-dimensional.
	\end{proposition}
	
	\begin{proof}
		Let $M=u(x_0)=\max_{\R^m}u$, and set $K=\{u\ge M-1\}$. If $K$ is unbounded, \cref{lem:recession} gives a unit vector $e$ with $u_e\ge0$ throughout $\R^m$. At the maximum point, $u_e(x_0)=0$. By \eqref{eq:jacobi}, $u_e\equiv0$, so $u$ is independent of one Euclidean variable.
		
		Repeat in the remaining variables whenever the corresponding superlevel set at height $M-1$ is unbounded. At each step the reduced function solves the same equation, since
		removing an invariant Euclidean direction does not change its
		Laplacian. The process must terminate before dimension zero, because
		$f>0$ excludes a constant solution. Thus, after an orthogonal change of coordinates, $u$ is the lift of a concave solution $w$ in $\R^r$, where $1\le r\le m$, such that
		\begin{equation}\label{eq:reduced-maximum}
			-\Delta w=f(w),\qquad \max_{\R^r}w=M,
			\qquad \{w\ge M-1\}\text{ is bounded}.
		\end{equation}
		
		Suppose $r\ge2$, and translate a maximum point of $w$ to the origin. Choose $R_0>0$ so that $w(R_0\theta)\le M-1$ for every $\theta\in\Sc^{r-1}$. Concavity on each ray implies that, for $R\ge R_0$,
		\begin{equation}\label{eq:linear-decay}
			w(R\theta)
			\le M+\frac{R}{R_0}\bigl(w(R_0\theta)-M\bigr)
			\le M-\frac{R}{R_0}.
		\end{equation}
		In particular, $w$ tends uniformly to $-\infty$ at least linearly as $|y|\to\infty$.
		
		Consider its spherical mean
		\[
		A(R)=\frac{1}{|\Sc^{r-1}|}
		\int_{\Sc^{r-1}}w(R\theta)\dd\theta.
		\]
		Concavity and the divergence theorem imply, respectively,
		\begin{equation}\label{eq:mean-concavity}
			A''(R)=\frac{1}{|\Sc^{r-1}|}
			\int_{\Sc^{r-1}}D^2w(R\theta)[\theta,\theta]\dd\theta\le0
		\end{equation}
		and
		\begin{equation}\label{eq:flux}
			A'(R)=\frac{1}{|\Sc^{r-1}|R^{r-1}}\int_{B_R}\Delta w\dd y
			=-\frac{1}{|\Sc^{r-1}|R^{r-1}}\int_{B_R}f(w)\dd y<0.
		\end{equation}
		
		We now show that $A'(R)\to0$. Given $\varepsilon>0$, the negative-tail assumption and \eqref{eq:linear-decay} give $R_\varepsilon$ such that
		\[
		0<f(w(y))\le\frac{\varepsilon}{|w(y)|}
		\le\frac{2R_0\varepsilon}{|y|}
		\qquad (|y|\ge R_\varepsilon).
		\]
		Consequently, since $r\ge2$,
		\[
		\int_{B_R}f(w)\dd y
		\le C_\varepsilon+
		\frac{2R_0\varepsilon|\Sc^{r-1}|}{r-1}R^{r-1}
		\qquad(R\ge R_\varepsilon).
		\]
		Together with \eqref{eq:flux}, this yields
		\[
		0<-A'(R)\le
		\frac{C_\varepsilon}{|\Sc^{r-1}|R^{r-1}}
		+\frac{2R_0\varepsilon}{r-1}.
		\]
		Now we first let $R\to\infty$, then let $\varepsilon\to0$, we obtain $A'(R)\to0$. However, \eqref{eq:mean-concavity} and \eqref{eq:flux} make $A'$ strictly negative and nonincreasing, so it cannot tend to zero. This reaches a contradiction, and thus $r=1$.
	\end{proof}

	This proves the one-dimensional rigidity statement whenever the maximum of the concave solution is attained.
	The proof also clarifies why that assumption matters: it forces each
	nondecreasing direction to be invariant and allows the spherical-mean
	argument to run after all such directions have been factored out. Without the attainment of the maximum, the same recession direction need not be
	invariant. The strict-monotonicity alternative then requires a
	completely different use of the PDE, which is the subject of the
	next section.
	
	\section{Rigidity when the supremum is not attained}\label{sec:descent}
	Suppose next that the supremum is not attained. Then every nonempty
	superlevel set is unbounded, and the recession lemma provides
	$u_e\ge0$ in some direction. If $u_e\equiv0$, one variable can
	immediately be removed. The new difficulty is the alternative
	$u_e>0$: there is no point at which the maximum principle forces
	invariance, but the positive derivative yields stability through
	the ground-state identity.
	
	We exploit this positivity by translating the solution in the
	increasing direction. The limiting profile is concave and
	independent of that direction, but at first it may be identically
	$+\infty$. Before applying elliptic compactness, one must therefore
	exclude loss of finiteness. Concavity ensures that divergence at a
	single transverse point propagates to the entire transverse space;
	monotonicity then places balls of arbitrarily large radii inside a
	fixed superlevel set, contradicting \cref{lem:ball}. The limit is
	thus finite, solves the same equation in one fewer variable, and
	inherits stability. The next proposition isolates this descent
	step, which is the main replacement for the missing maximum point.
	
	\begin{proposition}\label{prop:descent}
		Suppose $m\ge2$, $f\in C^1(\R)$, $f>0$, $f'>0$, and $\mu_c>0$ for every $c\in\R$. If a concave entire solution $u$ of \eqref{eq:semilinear} satisfies $u_e>0$ in $\R^m$ for some $e\in\Sc^{m-1}$, then the same equation admits a stable concave entire solution in $\R^{m-1}$.
	\end{proposition}
	
	\begin{proof}
		After a rotation write $x=(t,y)\in\R\times\R^{m-1}$ with $u_t>0$, and define
		\[
		\ell(y)=\lim_{t\to+\infty}u(t,y)\in\R\cup\{+\infty\}.
		\]
		We separate finiteness, compactness, and the passage of stability to the lower-dimensional factor.
		
		\smallskip
		\noindent\emph{Finiteness of the limit.}
		Concavity implies that $\ell$ is either finite everywhere or identically $+\infty$. Indeed, if $\ell(y_0)<\infty$ and $\ell(y_1)=+\infty$, put $y_2=2y_0-y_1$. For $t\ge0$,
		\[
		u(t,y_0)\ge\tfrac12u(t,y_1)+\tfrac12u(t,y_2)
		\ge\tfrac12u(t,y_1)+\tfrac12u(0,y_2),
		\]
		which diverges as $t\to\infty$, a contradiction.
		
		Suppose now that $\ell\equiv+\infty$. Fix $c\in\R$ and $R>0$. For each $y$ in the compact ball $\{|y|\le R\}$ there is a time when $u(t,y)>c$. Continuity in $y$, a finite covering of that ball, and monotonicity in $t$ give a common $T$ with $u(T,y)>c$ for all $|y|\le R$. Hence
		\[
		B_R((T+R+1,0))\subset\{u>c\}.
		\]
		Since $R$ is arbitrary, this contradicts \cref{lem:ball}. Therefore $\ell$ is finite everywhere.
		
		\smallskip
		\noindent\emph{The lower-dimensional equation.}
		The finite concave function $\ell$ is continuous on $\R^{m-1}$. Choose $t_j\uparrow\infty$, and set
		\[
		u_j(s,y)=u(t_j+s,y).
		\]
		The functions $u_j$ increase pointwise to the continuous function
		$\ell(y)$. Dini's theorem gives uniform convergence on compact subsets
		of $\R^m$. To pass to the equation, fix $K\Subset K'\Subset\R^m$
		and choose $p>m$. Both $u_j$ and $f(u_j)$ are uniformly bounded
		on $K'$. Interior $W^{2,p}$ estimates give uniform $C^{1,\alpha}$
		bounds on $K$, for some $\alpha>0$. The local Lipschitz continuity
		of $f$ then makes $f(u_j)$ uniformly $C^{0,\alpha}$ on smaller
		compact sets. Schauder estimates give local uniform $C^{2,\alpha}$
		bounds, after possibly reducing $\alpha$. Consequently
		$u_j\to\ell$ in $C^2_{\mathrm{loc}}$ (since the limit is unique by
		pointwise convergence), and hence
		\begin{equation}\label{eq:limit-equation}
			-\Delta_y\ell=f(\ell),\qquad D_y^2\ell\le0
			\qquad\text{in }\R^{m-1}.
		\end{equation}
		
		\smallskip
		\noindent\emph{Stability in one fewer variable.}
		Each $u_j$ is stable by \eqref{eq:groundstate}, with positive Jacobi field $\partial_su_j$. Passing to the limit in the quadratic form for each compactly supported test function first makes the lifted function $(s,y)\mapsto\ell(y)$ stable in $\R^m$. To obtain stability in $\R^{m-1}$, take
		\[
		\eta_L(s,y)=\chi(s/L)\zeta(y),
		\qquad 0\ne\chi\in\cc(\R),\quad
		\zeta\in\cc(\R^{m-1}).
		\]
		Separation of the integrals gives
		\begin{equation}\label{eq:stable-product}
			0\le L\|\chi\|_2^2Q_\ell[\zeta]
			+L^{-1}\|\chi'\|_2^2\|\zeta\|_2^2.
		\end{equation}
		After division by $L\|\chi\|_2^2$ and passage to $L\to\infty$, we obtain $Q_\ell[\zeta]\ge0$, as required.
	\end{proof}

	We now combine the attained-maximum proposition with the
	stability-preserving descent. This induction does not classify
	arbitrary stable entire solutions: it only reaches stable solutions
	that are also concave, so the lower-dimensional rigidity theorem can
	be applied before instability is invoked.
	\begin{proof}[Proof of \cref{thm:general}]
		We argue by induction on $n$. The assertion is immediate when $n=1$. Assume it has been proved in every lower dimension, and let $u$ be a concave solution in $\R^n$.
		
		If $u$ attains its supremum, \cref{prop:max} applies. Otherwise every nonempty superlevel set is unbounded, and \cref{lem:recession} gives $e\in\Sc^{n-1}$ with $u_e\ge0$. By \eqref{eq:jacobi}, either $u_e\equiv0$ or $u_e>0$ everywhere. In the first case, $u$ is the lift of a concave solution in $\R^{n-1}$, and induction gives one-dimensionality. In the second case, \cref{prop:descent} gives a stable concave solution $\ell$ in $\R^{n-1}$. Induction makes $\ell$ one-dimensional, contradicting \cref{lem:unstable}. This completes the induction.
	\end{proof}
	
	To deduce \cref{thm:Liouville}, we apply \cref{thm:general} to $f(s)=e^{2s}$. The resulting profile $u(x)=h(x\cdot e)$ satisfies $-h''=e^{2h}$. Its first integral is $(h')^2+e^{2h}=\lambda^2$ for some $\lambda>0$.
	With $v=e^{-h}$, this becomes $v''=\lambda^2v$ and
	$\lambda^2v^2-(v')^2=1$. Positivity of $v$ on the whole line gives
	$v(s)=\lambda^{-1}\cosh(\lambda(s-a))$, which is exactly \eqref{eq:form}.

	\subsection{Sharpness of the negative-tail hypothesis}\label{sec:scope}
	The order $o(|s|^{-1})$ in \cref{thm:general} is a genuine borderline for the stated hypotheses: at the inverse-linear order one can have strictly concave, non-one-dimensional solutions, even when all the other assumptions are retained.
	
	\begin{example}\label{ex:critical-tail}
		For each $n\ge2$, set
		\begin{equation}\label{eq:critical-example}
			u(x)=-\sqrt{1+|x|^2},\qquad
			f_n(s)=
			\begin{cases}
				\displaystyle\frac{n-1}{-s}+\frac{1}{(-s)^3},&s\le-1,\\[6pt]
				\displaystyle n\exp\!\left(\frac{n+2}{n}(s+1)\right),&s\ge-1.
			\end{cases}
		\end{equation}
		At $s=-1$ both formulas take the value $n$ and have first derivative $n+2$. Thus $f_n\in C^1(\R)$, with $f_n>0$, $f_n'>0$, and $\mu_c>0$ for all $c\in\R$, the last property following from the positive exponential tail. However,
		\[
		\lim_{s\to-\infty}|s|f_n(s)=n-1>0.
		\]
		Direct differentiation yields
		\[
		D^2u=-\frac{I}{\sqrt{1+|x|^2}}
		+\frac{x\otimes x}{(1+|x|^2)^{3/2}}<0,
		\]
		with radial eigenvalue $-(1+|x|^2)^{-3/2}$ and $n-1$ tangential eigenvalues $-(1+|x|^2)^{-1/2}$. Moreover,
		\[
		-\Delta u=\frac{n-1}{\sqrt{1+|x|^2}}
		+\frac{1}{(1+|x|^2)^{3/2}}=f_n(u).
		\]
		Hence $u$ is a smooth, strictly concave solution which is not one-dimensional. In particular, replacing $f(s)=o(|s|^{-1})$ as $s\to-\infty$ by $f(s)=O(|s|^{-1})$ in \cref{thm:general} would make the conclusion false, even for solutions attaining their maximum.
	\end{example}
	
	\section{A planar rigidity proof via a \texorpdfstring{$P$}{P}-function}\label{sec:planar}
	The preceding argument uses monotonicity of $f$ to create the contradiction with one-dimensional instability. In the plane, a superharmonic energy function implies a different route that requires only \eqref{eq:planar-assumption} and positivity of $f$.
	
	\begin{proof}[Proof of \cref{thm:planar}]
		Choose a primitive $F$ of $f$ bounded below on $\R$, and set
		\[
		P=\frac12|\nabla u|^2+F(u).
		\]
		The identity below is valid in the sense of distributions; the
		stated $C^1$ regularity of $f$, together with interior elliptic
		regularity, justifies all its terms. The equation and the
		Bochner identity give
		\begin{align}
			\Delta P
			&=|D^2u|^2+\nabla u\cdot\nabla\Delta u
			+f(u)\Delta u+f'(u)|\nabla u|^2\notag\\
			&=|D^2u|^2-f(u)^2
			=|D^2u|^2-(\Delta u)^2
			=-2\det D^2u\le0.\label{eq:P-laplacian}
		\end{align}
	The last inequality follows from the concavity assumption $D^2u\le0$ in dimension two. The integrability assumption makes $F$, and hence $P$, bounded
		below. The Liouville property for superharmonic functions bounded
		below on $\R^2$ now shows that $P$ is constant.
		
		Differentiating the constant-energy identity yields
		\begin{equation}\label{eq:Pgrad}
			(D^2u+f(u)I)\nabla u=0.
		\end{equation}
		The open set $\Omega=\{\nabla u\ne0\}$ is nonempty, since $f>0$ excludes a constant solution. On $\Omega$, let $\nu=\nabla u/|\nabla u|$. By \eqref{eq:Pgrad}, $\nu$ is a Hessian eigenvector with eigenvalue $-f(u)$. The trace of the Hessian is also $-f(u)$, so the other eigenvalue is zero. Therefore
		\begin{equation}\label{eq:rankone}
			D^2u=-f(u)\nu\otimes\nu,
			\qquad
			D\nu=\frac{(I-\nu\otimes\nu)D^2u}{|\nabla u|}=0
			\quad\text{on }\Omega.
		\end{equation}
		
		Fix $p\in\Omega$, and let $\Omega_p$ be its connected component. The vector $\nu$ is constant on $\Omega_p$. Choose a unit vector $\tau\perp\nu(p)$, and let $I$ be the component containing zero of $\{s:p+s\tau\in\Omega_p\}$. On $I$, \eqref{eq:rankone} gives
		\[
		\frac{\dd}{\dd s}\nabla u(p+s\tau)=D^2u(p+s\tau)\tau=0.
		\]
		Thus the gradient on this interval is the fixed nonzero vector $\nabla u(p)$. At a finite endpoint it would remain nonzero by continuity, so a neighborhood of that endpoint would lie in $\Omega$ and in the same connected component. This would extend $I$, a contradiction. Hence $I=\R$: through every regular point passes a complete regular level line.
		
		Two such lines cannot be nonparallel. Otherwise they intersect at a point with nonzero gradient perpendicular to both line directions, which is impossible in the plane. All regular level lines are therefore parallel. The gradient is parallel to one fixed vector $e$ on $\Omega$ and is zero outside $\Omega$, so every derivative perpendicular to $e$ vanishes globally. Thus $u(x)=h(x\cdot e)$.
	\end{proof}
	
	The preceding proof uses a Liouville property for bounded-below
	superharmonic functions which is specific to the plane. It is useful
	to explain why this very short argument does not replace the
	all-dimensional dimension-reduction method.
	
	\begin{remark}
		The constancy of a bounded-below superharmonic function is specific to dimensions at most two. In higher dimensions, \eqref{eq:P-laplacian} is replaced by $\Delta P=-2\sigma_2(D^2u)\le0$, but this sign and the lower bound alone do not force constancy. This is why the planar argument does not replace the descent used in \cref{thm:general}.
	\end{remark}

	\section{Quasiconcave Liouville solutions: two routes to classification}\label{sec:quasiconcave}
	
	In this section, we prove the known classification in \cref{thm:quasiconcave} in two ways. The first combines a new maximum-attainment lemma with the earlier bounded-above classification; the second gives an independent PDE proof.
	
	\subsection{Attainment of the maximum}
	
	The main obstacle in applying the bounded-above Liouville classification
	is that quasiconcavity does not itself give a finite upper bound.
	The next lemma eliminates both possibilities of a supremum escaping
	to infinity: it can be neither infinite nor finite but unattained.
	The proof uses only convexity of superlevel sets, the positive
	linearized potential, and the parabolicity of the plane.
	
	\begin{lemma}
		\label{ql:lem:max}
		Every quasiconcave entire solution of the planar Liouville equation has a
		finite, attained maximum.
	\end{lemma}
	
	\begin{proof}
		Suppose that $M:=\sup_{\R^2}u$ is either $+\infty$ or finite but not
		attained. Choose $x_j$ with $u(x_j)\to M$. Necessarily $|x_j|\to\infty$;
		passing to a subsequence, assume that $x_j/|x_j|\to e\in S^1$.
		
		Fix $x\in\R^2$ and $t>0$. For all large $j$, $u(x_j)>u(x)$ and
		$0<t/|x_j|<1$. Set
		\begin{equation}\label{ql:eq:convex-point}
			z_j=\left(1-\frac{t}{|x_j|}\right)x
			+\frac{t}{|x_j|}x_j
			=x+t\frac{x_j}{|x_j|}-\frac{t}{|x_j|}x.
		\end{equation}
		Quasiconcavity gives
		\[
		u(z_j)\ge\min\{u(x),u(x_j)\}=u(x).
		\]
		Since $z_j\to x+te$, continuity yields $u(x+te)\ge u(x)$ for every
		$x$ and $t>0$. Hence $w:=u_e\ge0$ in $\R^2$. By elliptic regularity,
		\begin{equation}\label{ql:eq:Jacobi}
			-\Delta w=2e^{2u}w.
		\end{equation}
		The strong minimum principle, applied to the weak linearized
		equation with locally bounded nonnegative potential, shows that
		either $w\equiv0$ or $w>0$.
		
		If $w>0$, the ground-state transform applied to \eqref{ql:eq:Jacobi}
		provides
		\begin{equation}\label{ql:eq:stability}
			\int_{\R^2}2e^{2u}\varphi^2\,dx
			\le\int_{\R^2}|\nabla\varphi|^2\,dx,
			\qquad \varphi\in C_c^\infty(\R^2).
		\end{equation}
		For $R>2$, take the logarithmic cutoff equal to $1$ on $B_R$,
		to $\log(R^2/|x|)/\log R$ on $B_{R^2}\setminus B_R$, and to zero outside
		$B_{R^2}$. It is admissible by $H^1$ approximation and has Dirichlet
		energy $2\pi/\log R$. Letting $R\to\infty$ in
		\eqref{ql:eq:stability} contradicts
		$\int_{B_1}e^{2u}\,dx>0$.
		
		Thus $w\equiv0$, and $u(x)=h(x\cdot e^\perp)$, where
		$-h''=e^{2h}$ on $\R$. Since $h''<0$, $h'$ is strictly decreasing.
		If $h'$ were positive on all of $\R$, then $h(s)\ge h(0)$ for $s\ge0$,
		so $h''(s)\le-e^{2h(0)}$, a contradiction for large $s$.
		Reflection rules out $h'<0$ on all of $\R$. Hence $h'$ vanishes once,
		and $h$ attains a finite maximum, contradicting the assumption on $M$.
	\end{proof}
	
	\subsection{A short deduction from the bounded-above classification}
	
	Once the maximum is known to exist, the result of
	Eremenko--Gui--Li--Xu leaves only the spherical family and the
	one-parameter nonradial family. Quasiconcavity rules out every
	nontrivial member of the latter family by inspecting its set of
	maximizers. We record this short deduction before turning to the
	independent PDE proof.
	
	\begin{proof}[First proof of \cref{thm:quasiconcave} (using \cite{EGLX})]
		By \cref{ql:lem:max}, $u$ is bounded from above. The classification of
		Eremenko--Gui--Li--Xu \cite[Theorem~1.6]{EGLX} states that, up to the
		natural Euclidean similarities and the scaling of the equation, such a
		solution is either spherical or belongs to
		\begin{equation}\label{ql:eq:family}
			u_b(s,t)=\log\frac{2e^s}{1+b^2+2be^s\cos t+e^{2s}},
			\qquad b\ge0.
		\end{equation}
		Set $A=\sqrt{1+b^2}$ and $s_0=\log A$. Then
		\begin{equation}\label{ql:eq:family-v}
			e^{-u_b(s,t)}=A\cosh(s-s_0)+b\cos t.
		\end{equation}
		If $b>0$, the points $(s_0,-\pi)$ and $(s_0,\pi)$ maximize $u_b$,
		whereas their midpoint $(s_0,0)$ does not: indeed, the values of
		$e^{-u_b}$ are respectively $A-b$ and $A+b$.
		This contradicts convexity of the maximum set of $u_b$. Therefore
		$b=0$, and \eqref{ql:eq:family} reduces to $u_0(s,t)=\log\sech s$.
		This proves the classification, and the converse follows because the
		superlevel sets in \eqref{ql:eq:spherical} and \eqref{ql:eq:oned} are
		respectively disks and strips (or empty sets).
	\end{proof}
	
	\subsection{An independent PDE proof}
	
	The preceding proof uses the known classification of bounded-above
	Liouville solutions. We now obtain the same conclusion directly from the equation.
	The additional ingredient is a gradient estimate derived from the
	geometry of convex superlevel sets.
	
	\begin{lemma}\label{ql:lem:gradient}
		Under the assumptions of \cref{thm:quasiconcave}, let
		$M=\max_{\R^2}u$. Then
		\begin{equation}\label{ql:eq:gradient}
			|\nabla u|^2+e^{2u}\le e^{2M}\qquad\text{in }\R^2.
		\end{equation}
	\end{lemma}
	
	\begin{proof}
		Put $P=|\nabla u|^2+e^{2u}$. At a regular point, let
		$g=|\nabla u|$, $\nu=\nabla u/g$, and let $\tau\perp\nu$ be a unit
		vector. The tangent line to the regular level set through this point
		supports the convex superlevel set. Consequently,
		$D^2u[\tau,\tau]\le0$. Since
		$u_{\nu\nu}+u_{\tau\tau}=-e^{2u}$, we obtain
		\begin{equation}\label{ql:eq:P-monotonicity}
			\nabla P\cdot\nu
			=2g\bigl(u_{\nu\nu}+e^{2u}\bigr)
			=-2g\,u_{\tau\tau}\ge0.
		\end{equation}
		
		Starting from a regular point $x_0$, let $\gamma:[0,L)\to\R^2$ be a
		maximal forward trajectory of
		\[
		\gamma'(s)=\frac{\nabla u(\gamma(s))}{|\nabla u(\gamma(s))|},
		\qquad \gamma(0)=x_0.
		\]
		Along $\gamma$, the function $P$ is nondecreasing and
		$\frac{d}{ds}u(\gamma(s))=|\nabla u(\gamma(s))|$.
		If $L<\infty$, the unit-speed curve converges to a critical point
		$x_*$ (otherwise it can be extended), and hence
		$P(x_0)\le P(x_*)=e^{2u(x_*)}\le e^{2M}$.
		If $L=\infty$, then
		\[
		\int_0^\infty|\nabla u(\gamma(s))|\,ds
		\le M-u(x_0)<\infty.
		\]
		There are $s_j\to\infty$ with $|\nabla u(\gamma(s_j))|\to0$.
		Therefore $P(x_0)\le\liminf_jP(\gamma(s_j))\le e^{2M}$.
		At a critical point, the desired bound holds because $u\le M$.
	\end{proof}
	
	\begin{proof}[Second proof of \cref{thm:quasiconcave} (independent PDE argument)]
		We retain \cref{ql:lem:max,ql:lem:gradient} but do not use
		\cite[Theorem~1.6]{EGLX}. Define the symmetric tensor
		\begin{equation}\label{ql:eq:stress}
			T=D^2u-\nabla u\otimes\nabla u
			+\frac12\bigl(|\nabla u|^2+e^{2u}\bigr)I_2.
		\end{equation}
		The equation $\Delta u=-e^{2u}$ gives
		\begin{equation}\label{ql:eq:stress-conservation}
			\operatorname{tr}T=0,\qquad \operatorname{div}T=0.
		\end{equation}
		Indeed, in index notation,
		\[
		\partial_jT_{ij}
		=\partial_i\Delta u-(\Delta u)u_i+e^{2u}u_i
		=\bigl(-2e^{2u}+e^{2u}+e^{2u}\bigr)u_i=0.
		\]
		By \eqref{ql:eq:gradient}, both $\nabla u$ and $e^{2u}$ are bounded.
		Also $-\Delta u_i=2e^{2u}u_i$ has globally bounded right-hand side.
		Uniform interior $W^{2,p}$ estimates for $u_i$ on balls of fixed
		radius, with $p>2$, show that $D^2u$ is globally bounded. Thus $T$ is
		bounded.
		
		Write
		\[
		T=\begin{pmatrix}a&b\\ b&-a\end{pmatrix}.
		\]
		The divergence-free identity becomes
		$a_x+b_y=0$, $b_x-a_y=0$, and implies
		$\Delta a=\Delta b=0$. Both components are bounded harmonic
		functions on $\R^2$, hence constant. Therefore $T$ is a constant
		trace-free matrix.
		
		Set $v=e^{-u}>0$. Direct calculation yields
		\begin{equation}\label{ql:eq:v-equation}
			v\Delta v-|\nabla v|^2=1,
		\end{equation}
		and
		\begin{equation}\label{ql:eq:T-v}
			T_{11}-T_{22}=-\frac{v_{xx}-v_{yy}}{v},
			\qquad T_{12}=-\frac{v_{xy}}{v}.
		\end{equation}
		
		\emph{Case 1: $T=0$.}
		Then $v_{xy}=0$ and $v_{xx}=v_{yy}$, so
		\[
		v(x,y)=A(x^2+y^2)+Bx+Cy+D.
		\]
		Equation \eqref{ql:eq:v-equation} implies
		$4AD-B^2-C^2=1$. Since $v>0$ on $\R^2$, $A>0$, and
		completing the square gives
		\[
		v(x)=A|x-x_0|^2+\frac1{4A}.
		\]
		Taking $\lambda=2A$ recovers \eqref{ql:eq:spherical}.
		
		\emph{Case 2: $T\ne0$.}
		After a rotation, write
		$T=\frac{\lambda^2}{2}\operatorname{diag}(1,-1)$, $\lambda>0$.
		By \eqref{ql:eq:T-v},
		\[
		v_{xy}=0,\qquad v_{xx}-v_{yy}=-\lambda^2v.
		\]
		From the first equation, we can write $v(x,y)=F(x)+G(y)$, and then further computation yields
		\[
		v(x,y)=\alpha\cos(\lambda x)+\beta\sin(\lambda x)
		+c e^{\lambda y}+d e^{-\lambda y}.
		\]
		Substitution in \eqref{ql:eq:v-equation} yields
		\begin{equation}\label{ql:eq:constraint}
			\lambda^2(4cd-\alpha^2-\beta^2)=1.
		\end{equation}
		In particular $cd>0$; positivity of $v$ forces $c,d>0$.
		Putting $\rho=(\alpha^2+\beta^2)^{1/2}$, $C=2\sqrt{cd}$, and
		translating the coordinates, we obtain
		\begin{equation}\label{ql:eq:normalform}
			v(x,y)=C\cosh\bigl(\lambda(y-y_0)\bigr)
			+\rho\cos\bigl(\lambda(x-x_0)\bigr),
			\qquad \lambda^2(C^2-\rho^2)=1.
		\end{equation}
		Suppose that $\rho>0$, then the minimum of $v$ is attained at a discrete
		sequence of points on the line $y=y_0$. Since the midpoint of two consecutive
		minimizers is not a minimizer, this contradicts the convexity of the set $\{v=\min v\}$. Therefore, $\rho=0$, and thus $C=1/\lambda$ and
		$v=\lambda^{-1}\cosh(\lambda(y-y_0))$, which is
		\eqref{ql:eq:oned}. This completes the proof.
	\end{proof}
	
	The two parts of the PDE argument have rather different scopes:
	maximum attainment and the gradient estimate survive for much more
	general nonlinearities, whereas the divergence-free tensor uses the
	exact exponential form. The following observation makes the
	separation explicit.
	
	\begin{remark}\label{ql:rem:extension}
		The maximum-attainment lemma extends to
		$-\Delta u=f(u)$ in $\R^2$ whenever $f\in C^1(\R)$,
		$f>0$ and $f'>0$: the same positive-Jacobi-field argument applies,
		and a one-dimensional entire solution of $-h''=f(h)$ necessarily
		attains its maximum. More generally, for any positive $f$ and any
		bounded-above quasiconcave solution, the gradient-flow calculation
		with $F'=f$ gives
		\[
		|\nabla u|^2+2F(u)\le 2F(\sup_{\R^2}u).
		\]
		When $F$ is bounded below on $(-\infty,\sup u]$, this also bounds
		$|\nabla u|$. The special role of the Liouville nonlinearity arises
		in \eqref{ql:eq:stress-conservation}: with
		\[
		T_f=D^2u-\nabla u\otimes\nabla u
		+\tfrac12\bigl(|\nabla u|^2+f(u)\bigr)I_2,
		\]
		one has
		\[
		\operatorname{div}T_f
		=\bigl(f(u)-\tfrac12f'(u)\bigr)\nabla u.
		\]
		Thus this particular tensor is divergence-free for $f(s)=Ce^{2s}$,
		but not for a general nonlinearity. The second proof establishes only
		the classification part of \cite{BEL}, not its additional quantitative
		statements concerning midpoint defects.
	\end{remark}
	
	\section*{Acknowledgment}
	C. Gui is partially supported by NSF grant DMS-2155183, Univer
	sity of Macau research grants CPG2023-00037-FST, CPG2024-00016-FST,
	MYRG-GRG2023-00139-FST-UMDF, SRG2023-00011-FST, a UMDF Pro
	fessorial Fellowship and Macao SAR FDCT 0003/2023/RIA1. Q. Li is sup
	ported by National Key R\&D Program of China (2022YFA1006900), the
	National Natural Science Fund of China for Youth Scholars of B class (No. 12522109) and National Science Fund of China General Program (No. 12471105). C. Wang is  supported by National Key R\&D Program of China (2022YFA1006900) and National Science Fund of China General Program (No. 12471106).  L. Xu is supported by NSFC No. 12171143.

	\section*{Declaration of generative AI and AI-assisted technologies}
	
	In the course of this work, the authors used ChatGPT (OpenAI) to assist with working out and refining technical details of the proofs, as well as with language editing and LaTeX typesetting. The research questions, proof frameworks, and main mathematical ideas originated with the authors. AI assistance contributed to the implementation of these ideas through interactive discussions with the authors. All AI-assisted arguments were critically examined and verified by the authors, who take full responsibility for the mathematical correctness, originality, references, and final content of the article.


\begin{thebibliography}{99}
		\bibitem{AAC}
		G.~Alberti, L.~Ambrosio and X.~Cabr\'e,
		\href{https://doi.org/10.1023/A:1010602715526}{\emph{On a long-standing conjecture of E.~De Giorgi: symmetry in 3D for general nonlinearities and a local minimality property}},
		Acta Appl. Math. \textbf{65} (2001), no.~1--3, 9--33.
		\bibitem{AC}
		L.~Ambrosio and X.~Cabr\'e,
		\href{https://doi.org/10.1090/S0894-0347-00-00345-3}{\emph{Entire solutions of semilinear elliptic equations in $\R^3$ and a conjecture of De Giorgi}},
		J. Amer. Math. Soc. \textbf{13} (2000), no.~4, 725--739.
		\bibitem{BEL}
		W.~Bergweiler, A.~Eremenko and J.~K.~Langley,
		\href{https://doi.org/10.15407/mag19.01.059}{\emph{On conformal metrics of constant positive curvature in the plane}},
		J. Math. Phys. Anal. Geom. \textbf{19} (2023), no.~1, 59--73.
		\bibitem{BG}
		B.~Bian and P.~Guan,
		\href{https://doi.org/10.1007/s00222-009-0179-5}{\emph{A microscopic convexity principle for nonlinear partial differential equations}},
		Invent. Math. \textbf{177} (2009), no.~2, 307--335.
		\bibitem{CF}
		L.~A.~Caffarelli and A.~Friedman,
		\href{https://doi.org/10.1215/S0012-7094-85-05221-4}{\emph{Convexity of solutions of semilinear elliptic equations}},
		Duke Math. J. \textbf{52} (1985), no.~2, 431--456.
		\bibitem{CaiLaiSurface}
		X.~Cai and M.~Lai,
		\href{https://doi.org/10.2140/pjm.2024.332.23}{\emph{Liouville equations on complete surfaces with nonnegative Gauss curvature}},
		Pacific J. Math. \textbf{332} (2024), no.~1, 23--38.
		\bibitem{CaiLaiSurvey}
		X.~Cai and M.~Lai,
		\href{https://doi.org/10.4208/jms.v57n3.24.10}{\emph{A survey of classification results for solutions to the Liouville equation}},
		J. Math. Study \textbf{57} (2024), no.~3, 398--424.
		\bibitem{CLZ}
		X.~Cai, M.~Lai and C.~Zhang,
		\href{https://doi.org/10.4310/MRL.260615144842}{\emph{An area growth estimate of the Liouville equation}},
		Math. Res. Lett. \textbf{33} (2026), no.~2, 567--595.
		\bibitem{CatinoMonticelli}
		G.~Catino and D.~D.~Monticelli,
		\href{https://doi.org/10.4171/JEMS/1484}{\emph{Semilinear elliptic equations on manifolds with nonnegative Ricci curvature}},
		J. Eur. Math. Soc. \textbf{28} (2026), no.~1, 359--392.
		\bibitem{CFFFS}
		H.~Chan, X.~Fern\'andez-Real, A.~Figalli, E.~Florit-Simon and J.~Serra,
		\emph{Stable solutions of the Allen--Cahn equation in dimension three are one-dimensional},
		preprint (2026), \href{https://arxiv.org/abs/2609.30194}{arXiv:2609.30194}.
		\bibitem{CL}
		W.~Chen and C.~Li,
		\href{https://doi.org/10.1215/S0012-7094-91-06325-8}{\emph{Classification of solutions of some nonlinear elliptic equations}},
		Duke Math. J. \textbf{63} (1991), no.~3, 615--622.
		\bibitem{CW}
		K.~S.~Chou and T.~Y.~H.~Wan,
		\href{https://doi.org/10.2140/pjm.1994.163.269}{\emph{Asymptotic radial symmetry for solutions of $\Delta u+e^u=0$ in a punctured disc}},
		Pacific J. Math. \textbf{163} (1994), no.~2, 269--276;
		\href{https://doi.org/10.2140/pjm.1995.171.589}{correction},
		\textbf{171} (1995), no.~2, 589--590.
		\bibitem{CFG}
		G.~Ciraolo, A.~Farina and M.~Gatti,
		\emph{The Liouville equation on Riemannian surfaces: the role of volume growth in classification and rigidity results},
		preprint (2026), \href{https://arxiv.org/abs/2607.19838}{arXiv:2607.19838}.
		\bibitem{CFP}
		G.~Ciraolo, A.~Farina and C.~C.~Polvara,
		\href{https://doi.org/10.4171/JEMS/1729}{\emph{Classification results, rigidity theorems and semilinear PDEs on Riemannian manifolds: A $P$-function approach}},
		J. Eur. Math. Soc. (2025), published online first.
		\bibitem{Dancer2008}
		E.~N.~Dancer,
		\href{https://doi.org/10.1016/j.anihpc.2006.12.001}{\emph{Finite Morse index solutions of exponential problems}},
		Ann. Inst. H. Poincar\'e Anal. Non Lin\'eaire \textbf{25} (2008), no.~1, 173--179.
		\bibitem{DF}
		E.~N.~Dancer and A.~Farina,
		\href{https://doi.org/10.1090/S0002-9939-08-09772-4}{\emph{On the classification of solutions of $-\Delta u=e^u$ on $\R^N$: stability outside a compact set and applications}},
		Proc. Amer. Math. Soc. \textbf{137} (2009), no.~4, 1333--1338.
		\bibitem{DG}
		E.~De Giorgi,
		\emph{Convergence problems for functionals and operators},
		in \emph{Proceedings of the International Meeting on Recent Methods in Nonlinear Analysis (Rome, 1978)},
		Pitagora, Bologna, 1979, 131--188.
		\bibitem{DPKW}
		M.~del Pino, M.~Kowalczyk and J.~Wei,
		\href{https://doi.org/10.4007/annals.2011.174.3.3}{\emph{On De Giorgi's conjecture in dimension $N\ge9$}},
		Ann. of Math. (2) \textbf{174} (2011), no.~3, 1485--1569.
		\bibitem{Emden}
		R.~Emden,
		\emph{Gaskugeln: Anwendungen der mechanischen W\"armetheorie auf kosmologische und meteorologische Probleme},
		B.~G.~Teubner, Leipzig, 1907.
		\bibitem{EGLX}
		A.~Eremenko, C.~Gui, Q.~Li and L.~Xu,
		\href{https://doi.org/10.4310/CAG.250221033911}{\emph{Rigidity results on Liouville equation}},
		Comm. Anal. Geom. \textbf{33} (2025), no.~1, 163--184.
		\bibitem{Farina}
		A.~Farina,
		\href{https://doi.org/10.1016/j.crma.2007.05.021}{\emph{Stable solutions of $-\Delta u=e^u$ on $\R^N$}},
		C. R. Math. Acad. Sci. Paris \textbf{345} (2007), no.~2, 63--66.
		\bibitem{GG}
		N.~Ghoussoub and C.~Gui,
		\href{https://doi.org/10.1007/s002080050196}{\emph{On a conjecture of De Giorgi and some related problems}},
		Math. Ann. \textbf{311} (1998), no.~3, 481--491.
		\bibitem{GL}
		C.~Gui and Q.~Li,
		\href{https://doi.org/10.1007/s00209-023-03369-5}{\emph{Some geometric inequalities related to Liouville equation}},
		Math. Z. \textbf{305} (2023), Paper No.~40.
		\bibitem{GLXX}
		C.~Gui, Q.~Li, C.~Xie and H.~Xu,
		\emph{Pressure estimates and classifications for steady $H^1_{\mathrm{loc}}$-solutions to two-dimensional incompressible Euler equations},
		preprint (2026), \href{https://arxiv.org/abs/2609.32647}{arXiv:2609.32647}.
		\bibitem{GXX}
		C.~Gui, C.~Xie and H.~Xu,
		\href{https://doi.org/10.1007/s00220-026-05632-2}{\emph{On a classification of steady solutions to two-dimensional Euler equations}},
		Comm. Math. Phys. \textbf{407} (2026), Paper No.~112.
		\bibitem{Hunter}
		C.~Hunter,
		\href{https://doi.org/10.1046/j.1365-8711.2001.04914.x}{\emph{Series solutions for polytropes and the isothermal sphere}},
		Mon. Not. R. Astron. Soc. \textbf{328} (2001), no.~3, 839--847.
		\bibitem{JL}
		D.~D.~Joseph and T.~S.~Lundgren,
		\href{https://doi.org/10.1007/BF00250508}{\emph{Quasilinear Dirichlet problems driven by positive sources}},
		Arch. Rational Mech. Anal. \textbf{49} (1973), 241--269.
		\bibitem{KL}
		N.~J.~Korevaar and J.~L.~Lewis,
		\href{https://doi.org/10.1007/BF00279844}{\emph{Convex solutions of certain elliptic equations have constant rank Hessians}},
		Arch. Rational Mech. Anal. \textbf{97} (1987), no.~1, 19--32.
		\bibitem{LX}
		Q.~Li and L.~Xu,
		\href{https://doi.org/10.1016/j.jfa.2025.110935}{\emph{Two new constant rank theorems}},
		J. Funct. Anal. \textbf{289} (2025), no.~4, Paper No.~110935, 18 pp.
		\bibitem{Liouville}
		J.~Liouville,
		\href{https://www.numdam.org/item/JMPA_1853_1_18__71_0/}{\emph{Sur l'\'equation aux diff\'erences partielles} $\frac{d^2\log\lambda}{du\,d\nu}\pm\frac{\lambda}{2a^2}=0$},
		J. Math. Pures Appl. (1) \textbf{18} (1853), 71--72.
		\bibitem{LLWWWW}
		Y.~Liu, T.~Luo, K.~Wang, J.~Wei, Y.~Wei and K.~Wu,
		\emph{Stable De Giorgi conjecture of the Allen--Cahn equation in $\R^3$},
		preprint (2026), \href{https://arxiv.org/abs/2609.21680}{arXiv:2609.21680}.
		\bibitem{Modica}
		L.~Modica,
		\href{https://doi.org/10.1002/cpa.3160380515}{\emph{A gradient bound and a Liouville theorem for nonlinear Poisson equations}},
		Comm. Pure Appl. Math. \textbf{38} (1985), no.~5, 679--684.
		\bibitem{Ou}
		Q.~Ou,
		\href{https://doi.org/10.1515/ans-2023-0220}{\emph{Classification results of Liouville equations and rigidity of Riemannian surfaces}},
		Adv. Nonlinear Stud. \textbf{26} (2026), no.~2, 670--676.
		\bibitem{PaynePhilippin}
		L.~E.~Payne and G.~A.~Philippin,
		\href{https://doi.org/10.1016/0022-0396(80)90086-8}{\emph{On maximum principles for a class of nonlinear second-order elliptic equations}},
		J. Differential Equations \textbf{37} (1980), no.~1, 39--48.
		\bibitem{Savin}
		O.~Savin,
		\href{https://doi.org/10.4007/annals.2009.169.41}{\emph{Regularity of flat level sets in phase transitions}},
		Ann. of Math. (2) \textbf{169} (2009), no.~1, 41--78.
		\bibitem{Weinberger}
		H.~F.~Weinberger,
		\href{https://doi.org/10.1007/BF00250469}{\emph{Remark on the preceding paper of Serrin}},
		Arch. Rational Mech. Anal. \textbf{43} (1971), no.~4, 319--320.
	\end{thebibliography}
\end{document}